\documentclass[12pt,a4]{amsart}
\usepackage[a4paper, left=28mm, right=28mm, top=28mm, bottom=34mm]{geometry}
\usepackage[all]{xy}
\usepackage{amsmath}
\usepackage{amscd}
\usepackage{amssymb}
\usepackage{amsthm}
\usepackage{color}
\usepackage{comment}
\usepackage{mathrsfs}
\theoremstyle{definition}
\newtheorem{thm}{Theorem}[section]
\newtheorem{lem}[thm]{Lemma}

\newtheorem{prop}[thm]{Proposition}

\theoremstyle{definition}

\newtheorem{rem}[thm]{Remark}

\newtheorem{defn}[thm]{Definition}

\numberwithin{equation}{section}

\def\F{{\mathbb F}}

\def\Q{{\mathbb Q}}

\def\Z{{\mathbb Z}}

\def\O{{\mathscr O}}

\def\Br{\mathop{\mathrm{Br}}\nolimits}

\def\Gal{\mathop{\mathrm{Gal}}\nolimits}

\def\Im{\mathop{\mathrm{Im}}\nolimits}

\def\Ker{\mathop{\mathrm{Ker}}\nolimits}

\def\Pic{\mathop{\mathrm{Pic}}\nolimits}

\numberwithin{equation}{section}

\title[Brauer groups of K3 surfaces in characteristic 2]{Finiteness of Brauer groups of K3 surfaces in characteristic 2}

\address{Department of Mathematics, Faculty of Science, 
Kyoto University, Kyoto 606-8502, Japan}
\email{kito@math.kyoto-u.ac.jp} 

\date{\today}



\author{Kazuhiro Ito}

\begin{document}
\maketitle
\begin{abstract}
For a $K3$ surface over a field of characteristic $2$ which is finitely generated over its prime subfield, we prove that the cokernel of the natural map from the Brauer group of the base field to that of the $K3$ surface is finite modulo the $2$-primary torsion subgroup. In characteristic different from $2$, such results were previously proved by A.\ N.\ Skorobogatov and Y.\ G.\ Zarhin. We basically follow their methods with an extra care in the case of superspecial  $K3$ surfaces using the recent results of W.\ Kim and K.\ Madapusi Pera on the Kuga-Satake construction and the Tate conjecture for $K3$ surfaces in characteristic $2$. \end{abstract}

\section{Introduction}\label{intro}
For a $K3$ surface $X$ over a field $k$, let $\Br(X):=H^2_{\mathrm{\acute{e}t}}(X, \mathbb{G}_{m})$ be the Brauer group of $X$ \cite{Grothendieck1}, \cite{Grothendieck2}. There is a natural map from the Brauer group $\Br(k)$ of $k$ to $\Br(X)$. The Brauer group $\Br(X)$ plays an important role in the study of the arithmetic of the $K3$ surface $X$; see \cite[Introduction]{Skorobogatov-Zarhin2}. \par
When $k$ is a field of characteristic $0$ which is finitely generated over its prime subfield, Skorobogatov and Zarhin proved that the cokernel of $\Br(k)\rightarrow\Br(X)$ is finite \cite[Theorem 1.2]{Skorobogatov-Zarhin2}. They also proved that when $k$ is a field of characteristic $p\geq3$ which is finitely generated over its prime subfield, the cokernel of $\Br(k)\rightarrow\Br(X)$ is finite modulo the $p$-primary torsion subgroup \cite[Theorem 1.3]{Skorobogatov-Zarhin}. \par
The purpose of this paper is to prove their results in characteristic $2$.
The main result of this paper is as follows. 

\begin{thm}\label{main2}
Let $k$ be a field of characteristic $2$ which is finitely generated over its prime subfield. For a $K3$ surface $X$ over $k$, the cokernel of the natural map $\Br(k)\rightarrow\Br(X)$ is finite modulo the $2$-primary torsion subgroup.
\end{thm}

We basically follow the methods of Skorobogatov and Zarhin with an extra care in the case of superspecial $K3$ surfaces.
 We first prove the following theorem on the Tate conjecture for $K3$ surfaces for $\ell$-adic and torsion coefficients.

\begin{thm}\label{main1}
Let $k$ be a field of characteristic $2$ which is finitely generated over its prime subfield, and $X$ a $K3$ surface over $k$. Let $k^{\mathrm{sep}}$ be a separable closure of $k$.
\begin{enumerate}
\item For $\ell \neq 2$, the $\Gal(k^{\mathrm{sep}}/k)$-module $H^2_{\mathrm{\acute{e}t}}(X_{k^{\mathrm{sep}}}, \Q_{\ell}{(1)})$ is semisimple.
\item For all but finitely many $\ell \neq 2$, the $\Gal(k^{\mathrm{sep}}/k)$-module $H^2_{\mathrm{\acute{e}t}}(X_{k^{\mathrm{sep}}}, \mu_{\ell})$ is semisimple and the cycle class map for $\ell$-torsion coefficients
$$\Pic(X)\rightarrow H^2_{\mathrm{\acute{e}t}}(X_{k^{\mathrm{sep}}}, \mu_{\ell})^{\Gal(k^{\mathrm{sep}}/k)}$$
is surjective.
\end{enumerate}
\end{thm}
Note that the cycle class map for $\Q_{\ell}$-coefficients (for any $\ell \neq 2$)
$$\Pic(X)\otimes_{\Z}{\Q_{\ell}}\simeq H^2_{\mathrm{\acute{e}t}}(X_{k^{\mathrm{sep}}}, \Q_{\ell}{(1)})^{\Gal(k^{\mathrm{sep}}/k)}$$
is surjective by the Tate conjecture for $K3$ surfaces in characteristic $2$ recently proved by Kim and Madapusi Pera \cite[Theorem A.1]{Kim-Madapusi}. (Previously, in characteristic $2$, the Tate conjecture was known only for ordinary $K3$ surfaces over finite fields \cite[Corollary 3.4]{Nygaard}.)

When the characteristic of $k$ is different from $2$, in \cite{Skorobogatov-Zarhin2}, \cite{Skorobogatov-Zarhin}, Skorobogatov and Zarhin proved Theorem \ref{main2} and Theorem \ref{main1} using the Kuga-Satake abelian varieties of $K3$ surfaces studied by Deligne, Pjatecki{\u\i}-{\v{S}}apiro, {\v{S}}afarevi{\v{c}}, and Andr$\mathrm{\acute{e}}$ in characteristic $0$ and by Rizov and Madapusi Pera in odd positive characteristics \cite{Madapusi}. They also used the results of Faltings (in characteristic $0$) and Zarhin (in positive characteristics) on the Galois action on torsion points of abelian varieties. (Zarhin's results are valid also in characteristic $2$; see \cite{Zarhin}.) Recently, in \cite{Kim-Madapusi}, Kim and Madapusi Pera extended the Kuga-Satake construction in \cite{Madapusi} to characteristic $2$, but they excluded polarized $K3$ surfaces $(X, \mathscr{L})$ satisfying ${\mathrm{ch_{dR}}}(\mathscr{L})\in F^2H^2_{\mathrm{dR}}(X)$. Such $K3$ surfaces are superspecial. (See Section \ref{sectionsuperspecial} for the definition of superspecial $K3$ surfaces.) Hence the methods of Skorobogatov and Zarhin work also for non-superspecial $K3$ surfaces in characteristic $2$. On the other hand, since the Tate conjecture was proved also for superspecial $K3$ surfaces in characteristic $2$ in \cite{Kim-Madapusi}, superspecial $K3$ surfaces have geometric Picard number $22$; it is easy to show Theorem \ref{main1} for superspecial (more generally, supersingular) $K3$ surfaces directly; see Lemma \ref{supersingular other proof}. Once Theorem \ref{main1} is established, we can prove Theorem \ref{main2} by the same methods as in \cite{Skorobogatov-Zarhin}. \par
The outline of this paper is as follows. In Section \ref{sectionsuperspecial}, we recall some well-known facts on superspecial $K3$ surfaces in characteristic $p>0$. We also prove that superspecial $K3$ surfaces in characteristic $2$ can be lifted to the ring of Witt vectors.  In Section \ref{proof of main1}, we prove Theorem \ref{main1}. In Section \ref{proof of main2}, we prove Theorem \ref{main2}. 

\begin{rem}
It seems an interesting but difficult problem to ask whether the $p$-primary torsion subgroup of the cokernel of $\Br(k)\rightarrow\Br(X)$ is finite, where  
 $k$ is an {\em infinite} field of characteristic $p>0$ which is finitely generated over its prime subfield and $X$ is a $K3$ surface over $k$. On the other hand, if $k$ is {\em finite}, the Brauer group $\Br(X)$ is finite by the Tate conjecture; see \cite[Theorem 4.1]{Milne}. (Note that Milne's result \cite[Theorem 4.1]{Milne} is stated under the additional condition $p \neq 2$ but the result is valid for any $p$; see \cite{Milne2}. See also \cite[footnote in p.\,674]{Liu}.) 
\end{rem}

\section{Superspecial $K3$ surfaces}\label{sectionsuperspecial}
In this section, we fix a prime number $p$, and an algebraically closed field $k$ of characteristic $p>0.$  
Recall that a \textit{$K3$ surface} $X$ over $k$ is a projective smooth surface with trivial canonical bundle and  $H^1(X,\O_X)=0$. Let $X$ be a $K3$ surface over $k$. Let $F^{\bullet}H^n_{\mathrm{dR}}(X)$ be the Hodge filtration on the de Rham cohomology $H^n_{\mathrm{dR}}(X)$ coming from the Hodge spectral sequence:
$$E^{ij}_{1}=H^j(X, \Omega^i_{X/k})\Rightarrow H^{n}_{\mathrm{dR}}(X),$$
which degenerates at $E_1$ \cite[Proposition 1.1]{Deligne}.
It is also equipped with the conjugate filtration $F^{\bullet}_{\mathrm{conj}}H^n_{\mathrm{dR}}(X)$ coming from the conjugate spectral sequence:
$$E^{ij}_{2}=H^i(X, \mathscr{H}^j(\Omega^{\bullet}_{X/k}))\Rightarrow H^{n}_{\mathrm{dR}}(X),$$
which degenerates at $E_2$ \cite[Section 1.1]{Ogus}. Hence $F^{\bullet}H^2_{\mathrm{dR}}(X)$ and $F^{\bullet}_{\mathrm{conj}}H^2_{\mathrm{dR}}(X)$ are of the following form:
$$0=F^3H^2_{\mathrm{dR}}(X)\subset F^2H^2_{\mathrm{dR}}(X) \subset F^1H^2_{\mathrm{dR}}(X) \subset F^0H^2_{\mathrm{dR}}(X)=H^2_{\mathrm{dR}}(X),$$ 
$$0=F^3_{\mathrm{conj}}H^2_{\mathrm{dR}}(X)\subset F^2_{\mathrm{conj}}H^2_{\mathrm{dR}}(X) \subset F^1_{\mathrm{conj}}H^2_{\mathrm{dR}}(X) \subset F^0_{\mathrm{conj}}H^2_{\mathrm{dR}}(X)=H^2_{\mathrm{dR}}(X).$$ 
For $i=0, 1, 2$, we have
$$F^iH^2_{\mathrm{dR}}(X)/F^{i+1}H^2_{\mathrm{dR}}(X)\simeq H^{2-i}(X, \Omega^i_{X/k}),$$
$$F^i_{\mathrm{conj}}H^2_{\mathrm{dR}}(X)/F^{i+1}_{\mathrm{conj}}H^2_{\mathrm{dR}}(X)\simeq H^i(X, \mathscr{H}^{2-i}(\Omega^{\bullet}_{X/k})).$$\par

 Let $W:=W(k)$ be the ring of Witt vectors of $k$. The Frobenius morphism $\sigma \colon k \rightarrow k$, $\sigma(x)=x^p$ induces a ring homomorphism on $W$. We also denote it by $\sigma \colon W \rightarrow W$. For a $K3$ surface $X$ over $k$, the crystalline cohomology $H^2_{\mathrm{cris}}(X/W)$ is a free $W$-module of rank $22$ equipped with a perfect pairing
$$\langle \ , \ \rangle \colon H^2_{\mathrm{cris}}(X/W) \times H^2_{\mathrm{cris}}(X/W) \rightarrow W.$$ The absolute Frobenius morphism on $X$ induces a $\sigma$-linear map $$\phi \colon H^2_{\mathrm{cris}}(X/W) \rightarrow H^2_{\mathrm{cris}}(X/W),$$
satisfying  
$\langle \phi(x), \phi(y) \rangle=p^2\langle x, y \rangle $
for any $x, y \in H^2_{\mathrm{cris}}(X/W).$ 
Let $$\pi \colon H^2_{\mathrm{cris}}(X/W) \rightarrow H^2_{\mathrm{dR}}(X)$$ be the reduction map. 
\begin{lem}[Mazur]\label{Mazur}
The following are satisfied. 
\begin{enumerate}
\item $\pi(\phi^{-1}pH^2_{\mathrm{cris}}(X/W))=F^1H^2_{\mathrm{dR}}(X)$ 
\item $\pi(\phi^{-1}p^2H^2_{\mathrm{cris}}(X/W))=F^2H^2_{\mathrm{dR}}(X) $
\item $\pi(p^{-1}(\Im\phi \cap pH^2_{\mathrm{cris}}(X/W)))=F^1_{\mathrm{conj}}H^2_{\mathrm{dR}}(X)$
\item $\pi(\Im\phi)=F^2_{\mathrm{conj}}H^2_{\mathrm{dR}}(X)$
\end{enumerate}
\end{lem}
\begin{proof}
See \cite[Theorem 8.26]{Berthelot-Ogus}. 
\end{proof}

We say a $K3$ surface $X$ over $k$ is \textit{supersingular} if the rank of $\Pic(X)$ is equal to $22$. The Tate conjecture for $K3$ surfaces \cite{Charles13}, \cite{Kim-Madapusi}, \cite{Madapusi}, \cite{Maulik} implies that $X$ is supersingular if and only if the height $h(X)$ of the formal Brauer group associated with $X$ is equal to $\infty$; see \cite[Corollaire 0.5]{Benoist}, \cite[Corollary 17.3.7]{Huybrechts}. We denote the dual lattice of $\Pic(X)$ by
\begin{align*}
\Pic(X)^{\vee} &:= {\mathrm{Hom}}_{\Z}(\Pic(X), \Z) \\
&\simeq \{\, x \in \Pic(X)\otimes_{\Z}{\Q} \mid (x, y) \in \Z \ {\mathrm{for\ any}}\ y \in \Pic(X)\,\},
\end{align*}
where $(\ , \ )$ is the intersection pairing on $\Pic(X)\otimes_{\Z}{\Q}$. 
For a supersingular $K3$ surface $X$ over $k$, it is well-known that $\Pic(X)$ is a $p$-elementary lattice. In other words, we have $p\Pic(X)^{\vee}\subset \Pic(X)$. Its discriminant is of the form $-p^{2a}$ for an integer $1\leq a \leq10$; see \cite[Proposition 1 in Section 8]{Rudakov-Shafarevich} and \cite[Theorem in Section 8]{Rudakov-Shafarevich}.
The integer $a$ is called the \textit{Artin invariant} of $X$. 

We fix the notation by the following commutative diagram:
$$\begin{CD}
                                                        &\Pic(X){\otimes}_{\Z}{W}   @>{\mathrm{ch_{cris}}}>>  &H^2_{\mathrm{cris}}(X/W) \\
                                                                             &@V{\pi_{\Pic}}VV                        &@V{\pi}VV                     \\
N:=\Ker({\mathrm{ch_{dR}}})\subset &\Pic(X){\otimes}_{\Z}{k} @>{\mathrm{ch_{dR}}}>>      &H^2_{\mathrm{dR}}(X)  \\
\end{CD}$$
We note that ${\mathrm{ch_{cris}}}$ is injective and preserves pairings. Hence we have ${\mathrm{dim}}_{k}N=a$. 

The following lemma is well-known. (For example, see \cite[p.268]{Ekedahl-Geer}.)
\begin{lem}\label{filtration}
For a supersingular $K3$ surface $X$ over $k$, the following are satisfied.
\begin{enumerate}
\item $p^{-1}{\mathrm{ch_{cris}}}({\pi^{-1}_{\Pic}}(N)\cap(1\otimes{\sigma})^{-1}({\pi^{-1}_{\Pic}}(N)))=\phi^{-1}pH^2_{\mathrm{cris}}(X/W)$ 
\item ${\mathrm{ch_{cris}}}((1\otimes{\sigma})^{-1}({\pi^{-1}_{\Pic}}(N)))=\phi^{-1}p^2H^2_{\mathrm{cris}}(X/W)$
\item $p^{-1}{\mathrm{ch_{cris}}}({\pi^{-1}_{\Pic}}(N)\cap(1\otimes{\sigma})({\pi^{-1}_{\Pic}}(N)))=p^{-1}(\Im\phi\cap pH^2_{\mathrm{cris}}(X/W))$
\item ${\mathrm{ch_{cris}}}((1\otimes{\sigma})({\pi^{-1}_{\Pic}}(N)))=\Im\phi$
\end{enumerate}
\end{lem}
\begin{proof}
We only sketch the proof. 
The main points are the following equalities:
\begin{itemize}
\item ${\mathrm{ch_{cris}}}({\pi^{-1}_{\Pic}}(N))=pH^2_{\mathrm{cris}}(X/W),$
\item $\phi\circ{\mathrm{ch_{cris}}}={\mathrm{ch_{cris}}}\circ(p\otimes\sigma).$
\end{itemize}
The first equality follows from $p(\Pic(X){\otimes}_{\Z}{W})^{\vee}\subset (\Pic(X){\otimes}_{\Z}{W})$ and $H^2_{\mathrm{cris}}(X/W) \subset (\Pic(X){\otimes}_{\Z}{W})^{\vee}$. 
We shall only prove $(1)$ since other equalities can be proved in the same way. For $x \in \mathrm{LHS}$, there exists $y \in (1\otimes{\sigma})^{-1}({\pi^{-1}_{\Pic}}(N))$ such that $px={\mathrm{ch_{cris}}}(y)$. Take $z \in H^2_{\mathrm{cris}}(X/W)$ such that ${\mathrm{ch_{cris}}}((1\otimes\sigma)(y))=pz$. We have 
$$p\phi(x)=\phi(px)=\phi({\mathrm{ch_{cris}}}(y))=p{\mathrm{ch_{cris}}}((1\otimes\sigma)(y))=p^2z.$$
Hence we have $\phi(x)=pz$ and $x \in \mathrm{RHS}$. 
Conversely, for $x \in \mathrm{RHS}$, there exists $y \in H^2_{\mathrm{cris}}(X/W)$ such that $\phi(x)=py$. Take $z \in {\pi^{-1}_{\Pic}}(K)$ such that $px={\mathrm{ch_{cris}}}(z)$. We have
 $$p{\mathrm{ch_{cris}}}((1\otimes\sigma)(z))= \phi({\mathrm{ch_{cris}}}(z))  =p\phi(x)=p^2y.$$
Hence we have ${\mathrm{ch_{cris}}}((1\otimes\sigma)(z))=py$ and $z \in (1\otimes{\sigma})^{-1}({\pi^{-1}_{\Pic}}(K)).$ So $x \in \mathrm{LHS}$.
\end{proof}
We recall the definition of \textit{superspecial} $K3$ surfaces from \cite[Section 2]{Ogus}.

\begin{defn}
A $K3$ surface $X$ over $k$ is \em{superspecial} if $$F^2H^2_{\mathrm{dR}}(X)=F^2_{\mathrm{conj}}H^2_{\mathrm{dR}}(X).$$ 
\end{defn}
Since $F^1H^2_{\mathrm{dR}}(X)=(F^2H^2_{\mathrm{dR}}(X))^{\perp}$ and $F^1_{\mathrm{conj}}H^2_{\mathrm{dR}}(X)=(F^2_{\mathrm{conj}}H^2_{\mathrm{dR}}(X))^{\perp}$ with respect to the canonical pairing
  $$H^2_{\mathrm{dR}}(X)\times H^2_{\mathrm{dR}}(X)\rightarrow k,$$ 
  a $K3$ surface $X$ is superspecial if and only if $F^{1}H^2_{\mathrm{dR}}(X)=F^{1}_{\mathrm{conj}}H^2_{\mathrm{dR}}(X)$.\par
The following proposition is presumably well-known. We include the proof for the readers' convenience.
\begin{prop}\label{superspecial}
A $K3$ surface $X$ over $k$ is superspecial if and only if $X$ is supersingular with Artin invariant $1$. 
\end{prop}
\begin{proof}
First, we assume that $X$ is superspecial. We have $F^1H^2_{\mathrm{dR}}(X)=F^1_{\mathrm{conj}}H^2_{\mathrm{dR}}(X)$ and their dimensions as $k$-vector spaces are equal to $21$. We see that $X$ is supersingular by \cite[Proposition 7.1]{van-katsura} and \cite[Lemma 9.6]{van-katsura}. We have 
\begin{align*}
& \hspace*{-1cm}\pi(p^{-1}{\mathrm{ch_{cris}}}({\pi^{-1}_{\Pic}}(N)\cap(1\otimes{\sigma})^{-1}({\pi^{-1}_{\Pic}}(N))))\\
&= \pi(p^{-1}{\mathrm{ch_{cris}}}({\pi^{-1}_{\Pic}}(N)\cap(1\otimes{\sigma})({\pi^{-1}_{\Pic}}(N))))
\end{align*}
 by $(1), (3)$ of Lemma \ref{Mazur} and $(1), (3)$ of Lemma \ref{filtration}. 
Hence we have 
\begin{align*}
& \hspace*{-1cm} {\mathrm{ch_{cris}}}({\pi^{-1}_{\Pic}}(N)\cap(1\otimes{\sigma})^{-1}({\pi^{-1}_{\Pic}}(N)))+p^2H^2_{\mathrm{cris}}(X/W) \\
&= {\mathrm{ch_{cris}}}({\pi^{-1}_{\Pic}}(N)\cap(1\otimes{\sigma})({\pi^{-1}_{\Pic}}(N)))+p^2H^2_{\mathrm{cris}}(X/W).
\end{align*}
We have
\begin{align}\label{equation}
&\hspace{-1cm}{\pi^{-1}_{\Pic}}(N)\cap(1\otimes{\sigma})^{-1}({\pi^{-1}_{\Pic}}(K))+p{\pi^{-1}_{\Pic}}(N)\\
&= {\pi^{-1}_{\Pic}}(N)\cap(1\otimes{\sigma})({\pi^{-1}_{\Pic}}(N))+p{\pi^{-1}_{\Pic}}(N)\nonumber
\end{align}
by ${\mathrm{ch_{cris}}}({\pi^{-1}_{\Pic}}(N))=pH^2_{\mathrm{cris}}(X/W)$ and the injectivity of ${\mathrm{ch_{cris}}}$. 
Taking the images of the both sides of (\ref{equation}) by $\pi_{\Pic}$, we have
$$N\cap(1\otimes{\sigma})^{-1}(N)=N\cap(1\otimes{\sigma})(N).$$
Hence we have
\begin{equation}\label{2.6}
N\cap(1\otimes{\sigma})(N)=(1\otimes{\sigma})(N\cap(1\otimes{\sigma})(N)).
\end{equation}
Then the Artin invariant of $X$ is equal to $1$ by \cite[Corollary 11.4]{van-katsura}. (In the notation of \cite{van-katsura}, we have $\sigma_0 = a$,  $U_{1}=N$, $U_{2}=U_{1}\cap(1\otimes{\sigma})(U_{1})$, and $U_{3}=U_{2}\cap (1\otimes{\sigma})(U_{2})$. By (\ref{2.6}), we have $U_{2}=U_{3}=0$. From this, we see that  $\sigma_0 =1$. Hence the Artin invariant of $X$ is $1$.)\par
Conversely, we assume that $X$ is supersingular with Artin invariant $1$. Since ${\mathrm{dim}}_{k}N=1$, the dimension of the  image of ${\mathrm{ch_{dR}}}$ is equal to $21$. Since the image of ${\mathrm{ch_{dR}}}$ is contained in $F^1H^2_{\mathrm{dR}}(X)\cap F^1_{\mathrm{conj}}H^2_{\mathrm{dR}}(X)$, we have $F^1H^2_{\mathrm{dR}}(X)=F^1_{\mathrm{conj}}H^2_{\mathrm{dR}}(X)$ by \cite[Proposition 7.1]{van-katsura}.
\end{proof}

Thanks to the Tate conjecture for $K3$ surfaces, we have the following lemma.
\begin{lem}[{\cite[Proposition 4.2]{Esnault-Oguiso}}]\label{ample}
Let $X$ be a $K3$ surface over $k$. Then there exists a primitive ample line bundle $\mathscr{L}$ such that ${\mathrm{ch_{dR}}}(\mathscr{L})$ is not contained in $F^2H^2_{\mathrm{dR}}(X)$.
\end{lem}
\begin{proof}
This result is proved by Esnault and Oguiso when $p\geq3$; see \cite[Proposition 4.2]{Esnault-Oguiso}. Thanks to the Tate conjecture for $K3$ surfaces in characteristic $2$ proved in \cite{Kim-Madapusi}, we can also prove it when $p=2$ as follows.
We have only to show that the canonical map
$$c_{1} \colon \Pic(X)\otimes_{\Z}{\F_{p}} \rightarrow H^1(X, \Omega^1_{X/k})$$
is not identically zero. (See the proof of \cite[Proposition 4.2]{Esnault-Oguiso}.) If the height of $X$ is not equal to $\infty$, this follows from \cite[Proposition 10.3]{van-katsura}. If the height of $X$ is $\infty$, then the rank of $\Pic(X)$ is equal to $22$ by the Tate conjecture \cite{Charles13}, \cite{Kim-Madapusi}, \cite{Madapusi}, \cite{Maulik}; see also \cite[Corollaire 0.5]{Benoist}, \cite[Corollary 17.3.7]{Huybrechts}. Since ${\mathrm{dim}}_{k}\Ker({\mathrm{ch_{dR}}})$ is equal to the Artin invariant $a$ with $1 \leq a \leq 10$, we have 
$${\mathrm{dim}}_{k}\Im({\mathrm{ch_{dR}}})=22-a  \geq12 >{\mathrm{dim}}_{k}F^2H^2_{\mathrm{dR}}(X)=1.$$
Hence $\Im({\mathrm{ch_{dR}}})$ is not contained in $F^2H^2_{\mathrm{dR}}(X)$, and the map $c_{1}$ is not identically zero. \end{proof}
The following result was proved by Deligne and Ogus when $p\geq3$ or $X$ is not superspecial. (For $K3$ surfaces in odd characteristics, see also Liedtke's lecture notes \cite[Theorem 2.9]{Liedtke}.) Using Lemma \ref{ample}, we can prove it also for superspecial $K3$ surfaces in characteristic $2$.

\begin{prop}[{\cite[Corollary 2.3]{Ogus}}]\label{lifting}
Let $X$ be a $K3$ surface over $k$. There is a projective smooth scheme $\mathscr{X}$ over $W$ such that $\mathscr{X}{\otimes_{{{W}}}}{k}\simeq X$.
\end{prop}
\begin{proof}
By Lemma \ref{ample}, the $K3$ surface $X$ has a primitive ample line bundle $\mathscr{L}$ such that ${\mathrm{ch_{dR}}}(\mathscr{L})$ is not contained in $F^2H^2_{\mathrm{dR}}(X)$. By \cite[Proposition 2.2.1]{Ogus}, the versal deformation space of $(X, \mathscr{L})$ is formally smooth over $W$. Hence by Grothendieck's algebraization theorem, there is a projective smooth scheme $\mathscr{X}$ over $W$ such that $\mathscr{X}{\otimes_{{{W}}}}{k}\simeq X$.
\end{proof}
\begin{rem}\label{superspecial defined over finite fields}
\rm{A superspecial $K3$ surface $X$ over $k$ is defined over $\overline{\F}_p$ and unique up to isomorphisms. (For $p>2$, see \cite[Corollary 7.14]{Ogus}. For $p=2$, see \cite[Section 11]{Rudakov-Shafarevich} and \cite[Theorem 1.1]{Dolgachev-Kondo}.)} Dolgachev and Kond$\mathrm{\overline{o}}$ calculated explicit defining equations of superspecial $K3$ surfaces in characteristic $2$; see \cite[Theorem 1.1]{Dolgachev-Kondo}. It is possible to prove Proposition \ref{lifting} directly for superspecial $K3$ surfaces in characteristic $2$ using the results of Dolgachev and Kond$\mathrm{\overline{o}}$. Sch${\mathrm{\overset{..}{u}}}$tt proved that a superspecial $K3$ surface $X$ has a model over $\F_p$ whose Picard number is equal to $21$; see \cite[Theorem 1]{Schutt}.
\end{rem}

\section{Proof of Theorem \ref{main1}}\label{proof of main1}
In this section, we shall prove Theorem \ref{main1}. First we make some preparation.
\begin{lem}\label{Picard group of K3}
Let $X$ be a proper scheme over a field $k$. Let $k^{\mathrm{al}}$ be an algebraic closure of $k$ and $k^{\mathrm{sep}}\subset k^{\mathrm{al}}$ the separable closure of $k$ inside $k^{\mathrm{al}}$. Assume that $X$ is geometrically connected over $k$ (i.e.\ $X_{k^{\mathrm{al}}}$ is connected) and $H^1(X, \O_{X})=0$. Then the natural map $$\Pic(X_{k^{\mathrm{sep}}}) \rightarrow \Pic(X_{k^{\mathrm{al}}})$$ is an isomorphism.

\end{lem}
\begin{proof}
Let $\underline{\Pic}_{X/k}$ be the Picard scheme of $X$. (See \cite[Section 8.2, Theorem 3]{BLR} for the representability of $\underline{\Pic}_{X/k}$.) Since $X$ is geometrically connected over $k$, we have the following exact sequences: 
$$0\rightarrow\Pic(X_{k^{\mathrm{sep}}})\rightarrow\underline{\Pic}_{X/k}(k^{\mathrm{sep}})\rightarrow\Br(k^{\mathrm{sep}}),$$
$$0\rightarrow\Pic(X_{k^{\mathrm{al}}})\rightarrow\underline{\Pic}_{X/k}(k^{\mathrm{al}})\rightarrow\Br(k^{\mathrm{al}}),$$
see \cite[Section 8.1, Proposition 4]{BLR}. Since the Brauer groups $\Br(k^{\mathrm{sep}})$, $\Br(k^{\mathrm{al}})$ are trivial, we have $\underline{\Pic}_{X/k}(k^{\mathrm{sep}})=\Pic(X_{k^{\mathrm{sep}}})$ and $\underline{\Pic}_{X/k}(k^{\mathrm{al}})=\Pic(X_{k^{\mathrm{al}}})$. Since we are assuming $H^1(X, \O_{X})=0$, the group scheme $\underline{\Pic}_{X/k}$ is zero dimensional and smooth, and so every connected component is $\mathrm{\acute{e}}$tale over $k$; \cite[Section 8.4, Theorem 1]{BLR}. It follows that $\underline{\Pic}_{X/k}(k^{\mathrm{sep}})\rightarrow\underline{\Pic}_{X/k}(k^{\mathrm{al}})$ is an isomorphism. Hence $\Pic(X_{k^{\mathrm{sep}}}) \rightarrow \Pic(X_{k^{\mathrm{al}}})$ is an isomorphism.
\end{proof}
\begin{lem}\label{supersingular other proof}
Let $X$ be a supersingular $K3$ surface over a field $k$ of characteristic $p>0$, namely the rank of $\Pic(X_{k^{\mathrm{al}}})$ is equal to $22$. 
\begin{enumerate}
\item For $\ell \neq p$, the $\Gal(k^{\mathrm{sep}}/k)$-module $H^2_{\mathrm{\acute{e}t}}(X_{k^{\mathrm{sep}}}, \Q_{\ell}{(1)})$ is semisimple.
\item For all but finitely many $\ell \neq p$, the $\Gal(k^{\mathrm{sep}}/k)$-module $H^2_{\mathrm{\acute{e}t}}(X_{k^{\mathrm{sep}}}, \mu_{\ell})$ is semisimple and the cycle class map for $\ell$-torsion coefficients
$$\Pic(X)\rightarrow H^2_{\mathrm{\acute{e}t}}(X_{k^{\mathrm{sep}}}, \mu_{\ell})^{\Gal(k^{\mathrm{sep}}/k)}$$
is surjective.
\end{enumerate}
\end{lem}
\begin{proof}
By Lemma \ref{Picard group of K3}, the rank of $\Pic(X_{k^{\mathrm{sep}}})$ is $22$. We have a $\Gal(k^{\mathrm{sep}}/k)$-equivariant isomorphism
$$\Pic(X_{k^{\mathrm{sep}}})\otimes_{\Z}{\Q_{\ell}}\simeq H^2_{\mathrm{\acute{e}t}}(X_{k^{\mathrm{sep}}}, \Q_{\ell}{(1)})$$
for any $\ell \neq p$. We take a finite Galois extension $k'/k$ such that the action of $\Gal(k^{\mathrm{sep}}/k')$ on $\Pic(X_{k^{\mathrm{sep}}})$ is trivial. Since the action of $\Gal(k^{\mathrm{sep}}/k)$ on $H^2_{\mathrm{\acute{e}t}}(X_{k^{\mathrm{sep}}}, \Q_{\ell}{(1)})$ factors through $\Gal(k'/k)$, the $\Gal(k^{\mathrm{sep}}/k)$-module $H^2_{\mathrm{\acute{e}t}}(X_{k^{\mathrm{sep}}}, \Q_{\ell}{(1)})$ is semisimple. This proves $(1)$. 

To study the torsion coefficients, we consider the cycle class map for $\Z_{\ell}$-coefficients
$$\Pic(X_{k^{\mathrm{sep}}})\otimes_{\Z}{\Z_{\ell}}\hookrightarrow H^2_{\mathrm{\acute{e}t}}(X_{k^{\mathrm{sep}}}, \Z_{\ell}{(1)})$$
for $\ell \neq p$. The cokernel of this map is a $\Z_{\ell}$-free module; see the exact sequence $(1)$ in \cite[p.11406]{Skorobogatov-Zarhin}. Hence we have a $\Gal(k^{\mathrm{sep}}/k)$-equivariant isomorphism
$$\Pic(X_{k^{\mathrm{sep}}})\otimes_{\Z}{\Z_{\ell}}\simeq H^2_{\mathrm{\acute{e}t}}(X_{k^{\mathrm{sep}}}, \Z_{\ell}{(1)})$$
for any $\ell \neq p$.
Consequently, we have
$$\Pic(X_{k^{\mathrm{sep}}})\otimes_{\Z}{\F_{\ell}}\simeq H^2_{\mathrm{\acute{e}t}}(X_{k^{\mathrm{sep}}}, \mu_{\ell})$$
for any $\ell \neq p$.
 Let $\ell$ be a prime number which does not divide $p[k':k]$.
Since the action of $\Gal(k^{\mathrm{sep}}/k)$ on $\Pic(X_{k^{\mathrm{sep}}})\otimes_{\Z}{\F_{\ell}}$ factors through $\Gal(k'/k)$, the $\Gal(k^{\mathrm{sep}}/k)$-module 
$H^2_{\mathrm{\acute{e}t}}(X_{k^{\mathrm{sep}}}, \mu_{\ell})$
is semisimple. 
Since we have the following exact sequence
$$0\rightarrow\Pic(X)\rightarrow\Pic(X_{k^{\mathrm{sep}}})^{\Gal(k^{\mathrm{sep}}/k)}\rightarrow\Br(k)$$
and $\Br(k)$ is a torsion abelian group, $\Pic(X)$ is a subgroup  of $\Pic(X_{k^{\mathrm{sep}}})^{\Gal(k^{\mathrm{sep}}/k)}$ of finite index.
Moreover, we have the following exact sequence
\begin{align*}
0& \rightarrow\Pic(X_{k^{\mathrm{sep}}})^{\Gal(k^{\mathrm{sep}}/k)}\otimes_{\Z}{\F_{\ell}}\rightarrow (\Pic(X_{k^{\mathrm{sep}}})\otimes_{\Z}{\F_{\ell}})^{\Gal(k^{\mathrm{sep}}/k)} \\
&\rightarrow H^1(k, \Pic(X_{k^{\mathrm{sep}}})). 
\end{align*}
We note that $H^1(k, \Pic(X_{k^{\mathrm{sep}}}))$ is finite; see \cite[Lemma 18.2.2]{Huybrechts}. Therefore, we have
\begin{align*}
\Pic(X)\otimes_{\Z}{\F_{\ell}}&\simeq\Pic(X_{k^{\mathrm{sep}}})^{\Gal(k^{\mathrm{sep}}/k)}\otimes_{\Z}{\F_{\ell}}\\
 &\simeq(\Pic(X_{k^{\mathrm{sep}}})\otimes_{\Z}{\F_{\ell}})^{\Gal(k^{\mathrm{sep}}/k)} \\
 &\simeq H^2_{\mathrm{\acute{e}t}}(X_{k^{\mathrm{sep}}}, \mu_{\ell})^{\Gal(k^{\mathrm{sep}}/k)}
\end{align*} for all but finitely many $\ell\neq p$. This proves $(2)$.
\end{proof}
Now, we shall prove Theorem \ref{main1}.

\begin{proof}[Proof of Theorem \ref{main1}]
Let $X$ be a $K3$ surface over a field $k$ of characteristic $2$ which is finitely generated over its prime subfield. \par
By Lemma \ref{supersingular other proof}, we may assume $X$ is not supersingular. By Proposition \ref{superspecial}, the $K3$ surface $X$ is not superspecial.
There exist a finite extension $k'$ of $k$ and an ample line bundle $\mathscr{L}$ on $X_{k'}$ such that its class in $\Pic(X_{k^{\mathrm{al}}})$ is primitive. We may assume $k'$ is a Galois extension of $k$ by Lemma \ref{Picard group of K3}. Since $X$ is not superspecial, we see that ${\mathrm{ch_{dR}}}(\mathscr{L})$ is not contained in $F^2H^2_{\mathrm{dR}}(X_{k'})$; see \cite[Corollary 1.4]{Ogus}.

By the results of Kim and Madapusi Pera \cite[Appendix A]{Kim-Madapusi}, after replacing $k'$ by its finite Galois extension, there exists an abelian variety $A$ over $k'$ such that, for any odd prime number $\ell \neq 2$, we have an embedding of $\Gal(k^{\mathrm{sep}}/k')$-modules
\begin{equation}\label{embedding}
P_{\mathscr{L}}H^2_{\mathrm{\acute{e}t}}(X_{k^{\mathrm{sep}}}, \Z_{\ell}{(1)})\hookrightarrow {\mathrm{End}}_{\Z_{\ell}}(H^1_{\mathrm{\acute{e}t}}(A_{k^{\mathrm{sep}}}, \Z_{\ell}))
\end{equation}
whose cokernel is torsion free. Here we denote the orthogonal complement to the image of $\mathscr{L}$ in $H^2_{\mathrm{\acute{e}t}}(X_{k^{\mathrm{sep}}}, \Z_{\ell}{(1)})$ by $P_{\mathscr{L}}H^2_{\mathrm{\acute{e}t}}(X_{k^{\mathrm{sep}}}, \Z_{\ell}{(1)})$. This result was proved in \cite[Proposition A.12 and Appendix A.14]{Kim-Madapusi}. (See also \cite[Theorem 5.17 (3)]{Madapusi}. For the construction of the $\ell$-adic sheaf $\mathbb{L}_{\ell}$ in \cite[Appendix A]{Kim-Madapusi}, see also \cite[Section 4.4]{Madapusi}. For the cokernel, see also \cite[p.\,11412]{Skorobogatov-Zarhin}.) The abelian variety $A$ is called the \textit{Kuga-Satake abelian variety} of the polarized $K3$ surface $(X_{k'}, \mathscr{L})$.

Using (\ref{embedding}), we can prove Theorem \ref{main1} by the same methods as in \cite{Skorobogatov-Zarhin}. We briefly sketch a proof for the readers' convenience. By (\ref{embedding}) and \cite[Proposition 4.1]{Skorobogatov-Zarhin}, we see that the $\Gal(k^{\mathrm{sep}}/k')$-module
$P_{\mathscr{L}}H^2_{\mathrm{\acute{e}t}}(X_{k^{\mathrm{sep}}}, \Z_{\ell}{(1)})\otimes_{\Z_{\ell}}{\Q_{\ell}}$
is semisimple for any odd prime number $\ell \neq 2$ and the $\Gal(k^{\mathrm{sep}}/k')$-module
$P_{\mathscr{L}}H^2_{\mathrm{\acute{e}t}}(X_{k^{\mathrm{sep}}}, \Z_{\ell}{(1)})\otimes_{\Z_{\ell}}{\F_{\ell}}$
is semisimple for all but finitely many $\ell \neq 2$. The proof of \cite[Proposition 4.1]{Skorobogatov-Zarhin} relies on the results on the Galois action on torsion points of abelian varieties due to Zarhin. (Zarhin's results are valid also in characteristic $2$; see \cite{Zarhin}.) 

Since we have a $\Gal(k^{\mathrm{sep}}/k')$-equivariant isomorphism
\[
H^2_{\mathrm{\acute{e}t}}(X_{k^{\mathrm{sep}}}, \Q_{\ell}{(1)})\simeq\Q_{\ell}\mathscr{L}\oplus (P_{\mathscr{L}}H^2_{\mathrm{\acute{e}t}}(X_{k^{\mathrm{sep}}}, \Z_{\ell}{(1)})\otimes_{\Z_{\ell}}{\Q_{\ell}}),
\]
the $\Gal(k^{\mathrm{sep}}/k')$-module $H^2_{\mathrm{\acute{e}t}}(X_{k^{\mathrm{sep}}}, \Q_{\ell}{(1)})$ is semisimple for any odd prime number $\ell \neq 2$. By \cite[Lemma 5]{Serre}, the $\Gal(k^{\mathrm{sep}}/k)$-module $H^2_{\mathrm{\acute{e}t}}(X_{k^{\mathrm{sep}}}, \Q_{\ell}{(1)})$ is also semisimple for any $\ell \neq 2$. We put $(\mathscr{L}, \mathscr{L})=2d$, where $(\ ,\ )$ is the intersection product on $\Pic(X_{k'})$ and $d \geq 1$ is a positive integer. When $2d$ is not divisible by $\ell$, we have a $\Gal(k^{\mathrm{sep}}/k')$-equivariant isomorphism
\[
H^2_{\mathrm{\acute{e}t}}(X_{k^{\mathrm{sep}}}, \mu_{\ell})\simeq\F_{\ell}\mathscr{L}\oplus (P_{\mathscr{L}}H^2_{\mathrm{\acute{e}t}}(X_{k^{\mathrm{sep}}}, \Z_{\ell}{(1)})\otimes_{\Z_{\ell}}{\F_{\ell}}).
\]
It follows that the $\Gal(k^{\mathrm{sep}}/k')$-module $H^2_{\mathrm{\acute{e}t}}(X_{k^{\mathrm{sep}}}, \mu_{\ell})$ is semisimple for all but finitely many $\ell \neq 2$. By \cite[Lemma 5]{Serre} again, if we further exclude the finitely many prime numbers dividing $[k:k']$, we see that the $\Gal(k^{\mathrm{sep}}/k)$-module $H^2_{\mathrm{\acute{e}t}}(X_{k^{\mathrm{sep}}}, \mu_{\ell})$ is semisimple for all but finitely many $\ell \neq 2$.

Let 
\[
T(X_{k^{\mathrm{sep}}})_{\ell} \subset H^2_{\mathrm{\acute{e}t}}(X_{k^{\mathrm{sep}}}, \Z_{\ell}{(1)})
\]
be the orthogonal complement to the image of $\Pic(X_{k^{\mathrm{sep}}})$ in $H^2_{\mathrm{\acute{e}t}}(X_{k^{\mathrm{sep}}}, \Z_{\ell}{(1)})$ for any $\ell \neq 2$. We have an inclusion of $\Gal(k^{\mathrm{sep}}/k')$-modules 
\[
T(X_{k^{\mathrm{sep}}})_{\ell}\hookrightarrow P_{\mathscr{L}}H^2_{\mathrm{\acute{e}t}}(X_{k^{\mathrm{sep}}}, \Z_{\ell}{(1)})
\]
whose cokernel is torsion free. By the Tate conjecture for $K3$ surfaces in characteristic $2$ \cite[Theorem A.1]{Kim-Madapusi}, \cite[Corollary 3.4]{Nygaard}, we have $(T(X_{k^{\mathrm{sep}}})_{\ell})^{\Gal(k^{\mathrm{sep}}/k')}=0$. 
For all but finitely many $\ell \neq 2$, we have 
\[(T(X_{k^{\mathrm{sep}}})_{\ell}\otimes_{\Z_{\ell}}\F_{\ell})^{\Gal(k^{\mathrm{sep}}/k')}=0\]
by \cite[Proposition 4.2]{Skorobogatov-Zarhin}. Let $d_X\in\Z$ be the discriminant of the lattice $\Pic(X_{k^{\mathrm{sep}}})$ endowed with the intersection product. When $d_X$ is not divisible by $\ell$, we have 
\[
H^2_{\mathrm{\acute{e}t}}(X_{k^{\mathrm{sep}}}, \Z_{\ell}{(1)})\simeq(\Pic(X_{k^{\mathrm{sep}}})\otimes_{\Z}\Z_{\ell})\oplus T(X_{k^{\mathrm{sep}}})_{\ell},
\]
and hence we have 
\[
H^2_{\mathrm{\acute{e}t}}(X_{k^{\mathrm{sep}}}, \mu_{\ell})\simeq(\Pic(X_{k^{\mathrm{sep}}})\otimes_{\Z}\F_{\ell})\oplus (T(X_{k^{\mathrm{sep}}})_{\ell}\otimes_{\Z_{\ell}}\F_{\ell}).
\]
It follows that for all but finitely many $\ell \neq 2$, we have 
\[
H^2_{\mathrm{\acute{e}t}}(X_{k^{\mathrm{sep}}}, \mu_{\ell})^{\Gal(k^{\mathrm{sep}}/k)}\simeq (\Pic(X_{k^{\mathrm{sep}}})\otimes_{\Z}\F_{\ell})^{\Gal(k^{\mathrm{sep}}/k)}.
\]
As in the proof of Lemma \ref{supersingular other proof}, we have 
\[
\Pic(X)\otimes_{\Z}{\F_{\ell}} \simeq (\Pic(X_{k^{\mathrm{sep}}})\otimes_{\Z}\F_{\ell})^{\Gal(k^{\mathrm{sep}}/k)}
\]
for all but finitely many $\ell \neq 2$. Therefore, we have
\[
\Pic(X)\otimes_{\Z}{\F_{\ell}} \simeq H^2_{\mathrm{\acute{e}t}}(X_{k^{\mathrm{sep}}}, \mu_{\ell}{(1)})^{\Gal(k^{\mathrm{sep}}/k)}
\]
for all but finitely many $\ell \neq 2$.
The proof of Theorem \ref{main1} is complete.\end{proof} 

\begin{rem}\label{superspecial point}
We use the notation of \cite[Appendix A]{Kim-Madapusi}. Let $M_{2d}$ be the moduli stack of quasi-polarized $K3$ surfaces and $\widetilde{M}_{2d}$ the $2$-fold $\mathrm{\acute{e}tale}$ cover of $M_{2d}$ as in \cite[Appendix A]{Kim-Madapusi}. Let $\widetilde{M}^{\mathrm{sm}}_{2d}$ be the smooth locus of $\widetilde{M}_{2d}$. We note that a superspecial $K3$ surface with an ample line bundle may not lie on the smooth locus $\widetilde{M}^{\mathrm{sm}}_{2d}$; see also \cite[Theorem 3.8]{Madapusi}. The reason why the results of Kim and Madapusi Pera do not apply to superspecial $K3$ surfaces in characteristic $2$ is we do not know whether we can extend the Kuga-Satake morphism 
\[
\iota^{\mathrm{KS}}: \widetilde{M}^{\mathrm{sm}}_{2d}\rightarrow \mathcal{S}(L_{2d})
\]
in \cite[Proposition A.12]{Kim-Madapusi} over all of $\widetilde{M}_{2d}$. Here $\mathcal{S}(L_{2d})$ is the $\textit{normal}$ integral model of the orthogonal Shimura variety $\mathrm{Sh}(L_{2d})$; see \cite[Appendix A.6]{Kim-Madapusi} for details.
\end{rem}

\begin{rem}\label{other proof}
Even if $X$ is a superspecial $K3$ surface over a field $k$ of characteristic $2$, there exist a finite separable extension $k'$ of $k$ and an ample line bundle $\mathscr{L}$ on $X_{k'}$ such that ${\mathrm{ch_{dR}}}(\mathscr{L})$ is not contained in $F^2H^2_{\mathrm{dR}}(X_{k'})$ and its class is primitive in $\Pic(X_{k^{\mathrm{al}}})$ by Lemma \ref{ample} and Lemma \ref{Picard group of K3}. Such a pair $(X_{k'}, \mathscr{L})$ lies on the smooth locus of the moduli stack $M_{2d}$ of quasi-polarized $K3$ surfaces; see \cite[Proposition 2.2]{Ogus}. Hence, after replacing $k'$ by its finite separable extension, there exists a Kuga-Satake abelian variety $A$ of $(X_{k'}, \mathscr{L})$  over $k'$ such that $(\ref{embedding})$ holds. We can use this fact to prove Theorem \ref{main1} for superspecial $K3$ surfaces in characteristic $2$ by the same methods as in \cite{Skorobogatov-Zarhin}. However, the proof in \cite{Skorobogatov-Zarhin} relies on the celebrated results on the Galois action on torsion points of abelian varieties due to Zarhin \cite{Zarhin}. We can prove Theorem \ref{main1} for superspecial $K3$ surfaces more easily as in Lemma \ref{supersingular other proof}.
\end{rem}

\section{Proof of Theorem \ref{main2}}\label{proof of main2}
Finally, we shall prove Theorem \ref{main2}. Once Theorem \ref{main1} is proved, the proof of Theorem \ref{main2} is the same as in \cite{Skorobogatov-Zarhin}. We only give a brief sketch of its proof. For details, see \cite{Skorobogatov-Zarhin2}, \cite{Skorobogatov-Zarhin}. \par
\begin{proof}[Proof of Theorem \ref{main2}]
From the Hochschild-Serre spectral sequence
$$E^{ij}_{2}=H^i(k, H^j_{\mathrm{\acute{e}t}}(X_{k^{\mathrm{sep}}}, \mathbb{G}_{m}))\Rightarrow H^n_{\mathrm{\acute{e}t}}(X, \mathbb{G}_{m}), $$
we have the following exact sequence:
$$0\rightarrow\Pic(X)\rightarrow\Pic(X_{k^{\mathrm{sep}}})^{\Gal(k^{\mathrm{sep}}/k)}\rightarrow\Br(k)\rightarrow  \Br_1(X)\rightarrow H^1(k, \Pic(X_{k^{\mathrm{sep}}})).$$
Here we put $$\Br_1(X):=\Ker(\Br(X)\rightarrow\Br(X_{k^{\mathrm{sep}}})^{\Gal(k^{\mathrm{sep}}/k)}).$$  Hence we have an injection
$$\Br(X)/\Br_1(X)\hookrightarrow \Br(X_{k^{\mathrm{sep}}})^{\Gal(k^{\mathrm{sep}}/k)}.$$
If we put $$\Br_0(X):=\Im(\Br(k)\rightarrow\Br(X)),$$ 
we have an injection
$$\Br_1(X)/\Br_0(X)\hookrightarrow H^1(k, \Pic(X_{k^{\mathrm{sep}}})).$$
We see that $H^1(k, \Pic(X_{k^{\mathrm{sep}}}))$ is finite; see \cite[Lemma 18.2.2]{Huybrechts}. Hence $\Br_1(X)/\Br_0(X)$ is finite. 
From the short exact sequence
$$0\rightarrow\Br_1(X)/\Br_0(X)\rightarrow\Br(X)/\Br_0(X)\rightarrow\Br(X)/\Br_1(X)\rightarrow0,$$
it suffices to prove that $\Br(X_{k^{\mathrm{sep}}})^{\Gal(k^{\mathrm{sep}}/k)}$ is finite modulo the $2$-primary torsion subgroup.\par
For an odd prime number $\ell \neq 2$, we have the following exact sequence:
$$0\rightarrow\Pic(X_{k^{\mathrm{sep}}})\otimes_{\Z}{\F_{\ell}}\rightarrow H^2_{\mathrm{\acute{e}t}}(X_{k^{\mathrm{sep}}}, \mu_{\ell})\rightarrow \Br(X_{k^{\mathrm{sep}}})[\ell] \rightarrow0.$$ 
Here we put $$\Br(X_{k^{\mathrm{sep}}})[\ell]:=\Ker(\Br(X_{k^{\mathrm{sep}}})\overset{\times\ell}{\rightarrow}\Br(X_{k^{\mathrm{sep}}})).$$

For all but finitely many $\ell \neq 2$,  the $\Gal(k^{\mathrm{sep}}/k)$-module 
$H^2_{\mathrm{\acute{e}t}}(X_{k^{\mathrm{sep}}}, \mu_{\ell})$
is semisimple by Theorem \ref{main1} $(2)$. 
Hence we have the following exact sequence:
\begin{align*}
0 &\rightarrow(\Pic(X_{k^{\mathrm{sep}}})\otimes_{\Z}{\F_{\ell}})^{\Gal(k^{\mathrm{sep}}/k)}\rightarrow H^2_{\mathrm{\acute{e}t}}(X_{k^{\mathrm{sep}}}, \mu_{\ell})^{\Gal(k^{\mathrm{sep}}/k)} \\
&\rightarrow (\Br(X_{k^{\mathrm{sep}}})[\ell])^{\Gal(k^{\mathrm{sep}}/k)}\rightarrow 0.
\end{align*}
For all but finitely many $\ell \neq 2$, the composition of the following maps  
$$\Pic(X)\otimes_{\Z}\F_{\ell} \rightarrow (\Pic(X_{k^{\mathrm{sep}}})\otimes_{\Z}{\F_{\ell}})^{\Gal(k^{\mathrm{sep}}/k)}\rightarrow H^2_{\mathrm{\acute{e}t}}(X_{k^{\mathrm{sep}}}, \mu_{\ell})^{\Gal(k^{\mathrm{sep}}/k)}$$ is surjective by Theorem \ref{main1} $(2)$. Hence we have $$(\Br(X_{k^{\mathrm{sep}}})[\ell])^{\Gal(k^{\mathrm{sep}}/k)}=0$$
for all but finitely many $\ell \neq 2$.   \par
On the other hand, for any odd prime number $\ell \neq 2$, the $\ell$-primary torsion subgroup of $\Br(X_{k^{\mathrm{sep}}})^{\Gal(k^{\mathrm{sep}}/k)}$ is finite by \cite[Proposition 2.5]{Skorobogatov-Zarhin2}. Here, we use Theorem \ref{main1} $(1)$ and the Tate conjecture for $X$ in $\Q_{\ell}$-coefficients \cite[Theorem A.1]{Kim-Madapusi}.\par
It is known that $\Br(X_{k^{\mathrm{sep}}})$ is a torsion abelian group; see \cite[Corollaire 1.5]{Grothendieck1} and \cite[Corollaire 2.2]{Grothendieck2}. We conclude that $\Br(X_{k^{\mathrm{sep}}})^{\Gal(k^{\mathrm{sep}}/k)}$ is finite modulo the $2$-primary torsion subgroup. \par
The proof of Theorem \ref{main2} is complete. 
\end{proof}

\section*{Acknowledgments}
The author is deeply grateful to my adviser, Tetsushi Ito, for his invaluable suggestions. He would like to thank Christian Liedtke for helpful comments. Moreover the author would like to express his gratitude to the anonymous referees for sincere remarks and comments.

\end{document}